\documentclass[12pt,reqno]{amsart}
\usepackage{times}

\newtheorem{thm}{Theorem}
\newtheorem{cor}[thm]{Corollary}
\newtheorem{lem}[thm]{Lemma}

\newcommand{\C}{{\mathbb C}}

\newcommand{\R}{{\mathbb R}}

\newcommand{\Z}{{\mathbb Z}}
\newcommand{\inc}{\int_\C}
\newcommand{\inr}{\int_\R}
\newcommand{\ccc}{\left(\frac2\pi\right)^{\frac14}}
\newcommand{\cccc}{\sqrt{\frac2\pi}}
\newcommand{\re}{{\rm Re\ }}

\begin{document}

\title[The Bargmann Transform]{Towards a Dictionary for the\\ Bargmann Transform}

\author{Kehe Zhu}
\address{Department of Mathematics and Statistics, State University of New York, Albany, NY 12222, USA.}
\email{kzhu@math.albany.edu}

\begin{abstract}
There is a canonical unitary transformation from $L^2(\R)$ onto the Fock space $F^2$, called the Bargmann
transform. The purpose of this article is to translate some important results and operators from the context of
$L^2(\R)$ to that of $F^2$. Examples include the Fourier transform, the Hilbert transform, Gabor frames, 
pseudo-differential operators, and the uncertainty principle.
\end{abstract}

\maketitle

\section{Introduction}

The Fock space $F^2$ is the Hilbert space of all entire functions $f$ such that
$$\|f\|^2=\inc|f(z)|^2\,d\lambda(z)<\infty,$$
where
$$d\lambda(z)=\frac1\pi e^{-|z|^2}\,dA(z)$$
is the Gaussian measure on the complex plane $\C$. Here $dA$ is ordinary area measure. The inner product
on $F^2$ is inherited from $L^2(\C,d\lambda)$.  The Fock space is a convenient setting for many problems in
functional analysis, mathematical physics, and engineering. A sample of early and recent work in these areas 
includes \cite{BCI, BC1, BC2, BC3, Guil, JPR, Landau, Lyu, LS1, LS2, OS, P, Seip1, Seip2, SW}. See \cite{Z2} 
for a recent summary of the mathematical theory of Fock spaces.

Another Hilbert space we need is $L^2(\R)=L^2(\R,dx)$. We will consider the Fourier transform, the Hilbert
transform, and several other operators and concepts on $L^2(\R)$. The books \cite{F1, F2, G1} are excellent
sources of information for these classical subjects.

The Bargmann transform $B$ is the operator from $L^2(\R)\to F^2$ defined by
$$Bf(z)=c\inr f(x)e^{2xz-x^2-(z^2/2)}\,dx,$$
where $c=(2/\pi)^{1/4}$. It is well known that $B$ is a unitary operator from $L^2(\R)$ onto $F^2$; see
\cite{F2, G1, Z2}. The easiest way to see this is outlined in the next section. Furthermore, the inverse of $B$ is 
also an integral operator, namely,
$$B^{-1}f(x)=c\inc f(z)e^{2x\overline z-x^2-(\overline z^2/2)}\,d\lambda(z).$$

The Bargmann transform is an old concept in mathematical analysis and mathematical physics.
In this article we attempt to establish a ``dictionary" between $L^2(\R)$ and $F^2$ that is based on the Bargmann 
transform. Thus we translate several important operators and concepts between these two spaces. It goes without 
saying that the dictionary is not complete, and it cannot be complete. Nevertheless, it covers some of the most 
important operators and concepts in Fourier and harmonic analysis, for example, the Fourier transform, the Hilbert 
transform, Gabor frames, the standard commutation relation, pseudo-differential operators, and the uncertainty principle.

We emphasize that we are presenting new forms of some classical results, although a few new results are also obtained
along the way. Some of these new forms look very appealing and appear to have been overlooked in the past. It is our 
hope that this article will generate some new interest in this old area of mathematical analysis.

I thank Hans Feichtinger and Bruno Torresani for their invitation to visit CIRM/Luminy in the fall semester of 2014.
This paper was motivated by discussions with several visitors during my stay at CIRM.

\section{Hermite polynomials}

The standard monomial orthonormal basis for $F^2$ is given by
$$e_n(z)=\sqrt{\frac1{n!}}\,z^n,\qquad n\ge0.$$
Thus the reproducing kernel of $F^2$ is 
$$K(z,w)=\sum_{n=0}^\infty e_n(z)\overline{e_n(w)}=\sum_{n=0}^\infty\frac{(z\overline w)^n}{n!}=e^{z\overline w}.$$ 
The normalized reproducing kernel of $F^2$ at the point $a$ is given by 
$$k_a(z)=e^{-\frac{|a|^2}2+z\overline a}.$$
Each $k_a$ is a unit vector in $F^2$.

To exhibit an orthonormal basis for $L^2(\R)$, recall that for any $n\ge0$ the function
$$H_n(x)=(-1)^ne^{x^2}\frac{d^n}{dx^n}e^{-x^2}$$
is called the nth Hermite polynomial. It is well known that the functions
$$h_n(x)=\frac c{\sqrt{2^nn!}}\,e^{-x^2}H_n(\sqrt2\,x),\qquad n\ge0,$$
form an orthonormal basis for $L^2(\R)$, where $c=(2/\pi)^{1/4}$ again. See \cite{Z2} for more information about
the Hermite functions.

\begin{thm}
For every $n\ge0$ we have $Bh_n=e_n$.
\label{1}
\end{thm}

\begin{proof}
See Theorem 6.8 of \cite{Z2}.
\end{proof}

\begin{cor}
The Bargmann transform is a unitary operator from $L^2(\R)$ onto $F^2$, and it maps the normalized Gauss function
$$g(z)=(2/\pi)^{1/4}e^{-x^2},$$
which is a unit vector in $L^2(\R)$, to the constant function $1$ in $F^2$.
\label{2}
\end{cor}

That $B$ maps the Gauss function to a constant is the key ingredient when we later translate Gabor frames with the 
Gauss window to analytic atoms in the Fock space.

\section{The Fourier transform}

There are several normalizations for the Fourier transform. We define the Fourier transform by
$$F(f)(x)=\frac1{\sqrt\pi}\inr f(t)e^{2ixt}\,dt.$$
It is well known that the Fourier transform acts as a bounded linear operator on $L^2(\R)$. In fact, Plancherel's
formula tells us that $F$ is a unitary operator on $L^2(\R)$, and its inverse is given by
$$F^{-1}(f)(x)=\frac1{\sqrt\pi}\inr f(t)e^{-2ixt}\,dt.$$

It is not at all clear from the definition that $F$ is bounded and invertible on $L^2(\R)$. There are also issues 
concerning convergence: it is not clear that the integral defining $F(f)$ converges in $L^2(\R)$ for $f\in L^2(\R)$. 
The situation will change dramatically once we translate $F$ to an operator on the Fock space. In other words, 
we show that, under the Bargmann transform, the operator $F:L^2(\R)\to L^2(\R)$ is unitarily equivalent to an 
extremely simple operator on the Fock space $F^2$.

\begin{thm}
The operator 
$$T=BFB^{-1}:F^2\to F^2$$ 
is given by $Tf(z)=f(iz)$ for all $f\in F^2$. Consequently, the operator 
$$T^{-1}=BF^{-1}B^{-1}:F^2\to F^2,$$ 
where $F^{-1}$ is the inverse Fourier transform, is given by $T^{-1}f(z)=f(-iz)$ for all $f\in F^2$.
\label{3}
\end{thm}

\begin{proof}
We consider a family of Fourier transforms as follows:
$$F_\sigma(f)(x)=\sqrt{\frac\sigma\pi}\inr f(t)e^{2i\sigma xt}\,dt,$$
where $\sigma$ is a real parameter with the convention that $\sqrt{-1}=i$.

For the purpose of applying Fubini's theorem in the calculations below, we assume that
$f$ is any polynomial (recall that the polynomials are dense in $F^2$, and under
the inverse Bargmann transform, they become the Hermite polynomials times the Gauss
function, which have very good integrability on the real line). For $c=(2/\pi)^{1/4}$ again, we have
\begin{eqnarray*}
F_\sigma(B^{-1}f)(x)&=&c\sqrt{\frac\sigma\pi}\inr e^{2i\sigma xt}\,dt\inc f(z)e^{2t\overline z-t^2-\frac{\overline z^2}2}
\,d\lambda(z)\\
&=&c\sqrt{\frac\sigma\pi}\inc f(z)e^{-\frac{\overline z^2}2}\,d\lambda(z)\inr e^{2t(i\sigma x+\overline z)-t^2}\,dt\\
&=&c\sqrt{\frac\sigma\pi}\inc f(z)e^{-\frac{\overline z^2}2+(i\sigma x+\overline z)^2}\,d\lambda(z)
\inr e^{-(t-i\sigma x-\overline z)^2}\,dt\\
&=&c\sqrt\sigma e^{-\sigma^2x}\inc f(w)e^{\frac{\overline w^2}2+2i\sigma x\overline w}\,d\lambda(w).
\end{eqnarray*}
It follows that
\begin{eqnarray*}
(BF_\sigma B^{-1}f)(z)&=&c^2\sqrt\sigma\inr e^{2xz-x^2-\frac{z^2}2-\sigma^2x^2}\,dx
\inc f(w)e^{\frac{\overline w^2}2+2i\sigma x\overline w}\,d\lambda(w)\\
&=&c^2\sqrt\sigma e^{-\frac{z^2}2}\inr e^{2xz-(1+\sigma^2)x^2}\,dx
\inc f(w)e^{\frac{\overline w^2}2+2i\sigma x\overline w}\,d\lambda(w)\\
&=&c^2\sqrt\sigma e^{-\frac{z^2}2}\inc f(w)e^{\frac{\overline w^2}2}\,d\lambda(w)
\inr e^{2x(z+i\sigma\overline w)-(1+\sigma^2)x^2}\,dx\\
&=&c^2\sqrt{\frac\sigma{\sigma^2+1}}e^{-\frac{z^2}2}\inc f(w)e^{\frac{\overline w^2}2}\,d\lambda(w)
\inr e^{2\cdot\frac{z+i\sigma\overline w}{\sqrt{\sigma^2+1}}\cdot t-t^2}\,dt\\
&=&c^2\sqrt{\frac{\sigma\pi}{\sigma^2+1}}e^{-\frac{z^2}2}\inc f(w)e^{\frac{\overline w^2}2+\frac{(z
+i\sigma\overline w)^2}{\sigma^2+1}}\,d\lambda(w)\\
&=&\sqrt{\frac{2\sigma}{\sigma^2+1}}e^{\left(\frac1{\sigma^2+1}-\frac12\right)z^2}
\inc f(w)e^{\left(\frac12-\frac{\sigma^2}{\sigma^2+1}\right)\overline w^2+
\frac{2i\sigma}{\sigma^2+1}\,z\overline w}\,d\lambda(w).
\end{eqnarray*}

In the case $\sigma=1$, we have $F_\sigma=F$, so that
$$(BFB^{-1}f)(z)=\inc f(w)e^{iz\overline w}\,d\lambda(w)=f(iz).$$
In the case $\sigma=-1$, we have
$$(BF^{-1}B^{-1}f)(z)=\inc f(w)e^{-iz\overline w}\,d\lambda(w)=f(-iz).$$
This completes the proof of the theorem.
\end{proof}

The theorem above is probably known to experts, but we have been unable to locate a reference for it. 
As a consequence of Theorem~\ref{3}, we obtain an alternative proof of Plancherel's formula.

\begin{cor}[Plancherel's formula]
The Fourier transform is a unitary operator on $L^2(\R)$: it is one-to-one, onto, and isometric 
in the sense that
$$\inr|F(f)|^2\,dx=\inr|f|^2\,dx$$
for all $f\in L^2(\R)$.
\label{4}
\end{cor}

Note that it is a little vague how the Fourier transform is defined for a function in $L^2(\R)$.
But there is absolutely no ambiguity for the unitary operator $f(z)\mapsto f(iz)$ on $F^2$. This operator
is clearly well defined for every $f\in F^2$ and it is clearly a unitary operator.

The following result is clear from our new representation of the Fourier transform on the Fock space,
because an entire function uniquely determines its Taylor coefficients.

\begin{cor}
For each $n\ge0$ the function $h_n$ is an eigenvector of the Fourier transform and the corresponding 
eigenvalue is $i^n$. Furthermore, the fixed points of the Fourier transform are exactly functions of the form
$$f(x)=\sum_{n=0}^\infty c_nh_{4n}(x),\qquad \{c_n\}\in l^2.$$
\label{5}
\end{cor}

To go one step further, we can completely determine the spectral properties of the Fourier transform
as a unitary operator on $L^2(\R)$.

\begin{cor}
For each $0\le k\le3$ let $X_k$ denote the closed subspace of $L^2(\R)$ spanned by the
Hermite functions $h_{k+4m}$, $m\ge0$, and let $P_k:L^2(\R)\to X_k$ be the orthogonal
projection. Then
$$L^2(\R)=\bigoplus_{k=0}^3 X_k,$$
and the corresponding spectral decomposition for the unitary operator $F:L^2(\R)\to L^2(\R)$
is given by
$$F=P_0+iP_1-P_2-iP_3.$$
\label{6}
\end{cor}

In particular, each $X_k$ is nothing but the eigenspace of the Fourier transform corresponding
to the eigenvalue $i^k$.

Finally in this section we mention that for any real $\theta$ the operator $U_\theta$ defined
by $U_\theta f(z)=f(e^{i\theta}z)$ is clearly a unitary operator on $F^2$. For $\theta=\pm\pi/2$
the resulting operators are unitarily equivalent to the Fourier and inverse Fourier transforms
on $L^2(\R)$. When $\theta$ is a rational multiple of $\pi$, the structure of $U_\theta$ is relatively 
simple. However, if $\theta$ is an irrational multiple of $\pi$, the structure of such an ``irrational
rotation operator" $U_\theta$ is highly nontrivial. 
It would be interesting to find out the operator on $L^2(\R)$ that is unitarily equivalent to $U_\theta$
via the Bargmann transform.

\section{Dilation, translation, and modulation operators}

For any positive $r$ we consider the dilation operator $D_r:L^2(\R)\to L^2(\R)$ defined by
$D_rf(x)=\sqrt rf(rx)$. It is obvious that $D_r$ is a unitary operator on $L^2(\R)$. Furthermore,
it follows from the definition of $F$, $F^{-1}$, and $F_r$ in the previous section that
$$(F_rF^{-1}f)(x)=\sqrt r(FF^{-1}f)(rx)=\sqrt rf(rx)=D_rf(x).$$
Therefore, each $F_r=D_rF$ is a unitary operator on $L^2(\R)$ with
$$F_r^{-1}=F^{-1}D_r^{-1}=F^{-1}D_{1/r},$$
that is,
$$F_r^{-1}f(x)=\frac1{\sqrt{\pi r}}\inr f(t/r)e^{-2ixt}\,dt=\sqrt{\frac r\pi}\inr f(t)e^{-2irxt}\,dt.$$
Combining this with the proof of Theorem~\ref{3}, we obtain the following form of the dilation 
operator $D_r$ on the Fock space.

\begin{thm}
For any $r>0$ let $T_r=BD_rB^{-1}$ on $F^2$. Then
$$T_rf(z)=\sqrt{\frac{2r}{1+r^2}}\inc f(-iw)e^{\left(\frac12-\frac{r^2}{1+r^2}\right)\overline w^2}
e^{\frac{2irz}{1+r^2}\overline w}\,d\lambda(w)$$
for all $f\in F^2$.
\label{7}
\end{thm}

In this case, the operator is much simpler on the space $L^2(\R)$ than on the space $F^2$.
Therefore, it is unlikely that the Fock space will be helpful in the study of dilation operators 
on $L^2(\R)$. Since the theory of wavelets depends on dilation in a critical way, we expect
the Fock space (and its associated complex analysis) to be of limited use for wavelet analysis.
See \cite{Coburn1} for an application of the Bargmann transform to the study of the 
wavelet localization operators.

The situation is completely different for Gabor analysis (see next section), where critical roles are 
played by the so-called translation and modulation operators, whose representations on $F^2$ via 
the Bargmann transform are much more useful. More specifically, for any real numbers $a$ and $b$ 
we define two unitary operators $T_a$ and $M_b$ on $L^2(\R)$ as follows:
$$T_af(x)=f(x-a),\qquad M_bf(x)=e^{2\pi bix}f(x).$$
It is traditional to call $T_a$ a translation operator and $M_b$ a modulation operator. See
\cite{G1} for more information on such operators.

To identify the equivalent form of $T_a$ and $M_b$ on the Fock space, we need the
classical Weyl operators on $F^2$. Recall that for any complex number $a$ the Weyl
operator $W_a$ on $F^2$ is defined by
$$W_af(z)=f(z-a)k_a(z)=f(z-a)e^{z\overline a-\frac{|a|^2}2},$$
where $k_a$ is the normalized reproducing kernel of $F^2$ at $a$. It is well known, and it follows
easily from a change of variables, that each $W_a$ is a unitary operator on $F^2$. See \cite{Z2}
for more information about the Weyl operators.

\begin{thm}
For any $a\in\R$ we have $BT_aB^{-1}=W_a$.
\label{8}
\end{thm}

\begin{proof}
For any polynomial $f$ in $F^2$ we have
$$T_aB^{-1}f(x)=\ccc\inc f(z)e^{-(x-a-\overline z)^2+(\overline z^2/2)}\,d\lambda(z),$$
and
\begin{eqnarray*}
BT_aB^{-1}f(z)&=&\cccc e^{\frac{z^2}2}\inr e^{-(x-z)^2}\,dx\inc f(u)e^{-(x-a-\overline u)^2
+\frac{\overline u^2}2}\,d\lambda(u)\\
&=&\cccc e^{\frac{z^2}2}\inc f(u)e^{\frac{\overline u^2}2}\,d\lambda(u)\inr e^{-(x-z)^2-(x-a-\overline u)^2}\,dx\\
&=&\cccc e^{-\frac{z^2}2}\inc f(u)e^{-\frac12\overline u^2-a^2-2a\overline u}I(u,z)\,d\lambda(u),
\end{eqnarray*}
where
\begin{eqnarray*}
I(u,z)&=&\inr e^{-2x^2+2xz+2ax+2x\overline u}\,dx\\
&=&e^{(z+a+\overline u)^2/2}\inr e^{-2[x+(z+a+\overline u)/2]^2}\,dx\\
&=&\sqrt{\frac\pi2}e^{(z+a+\overline u)^2/2}.
\end{eqnarray*}
It follows that
\begin{eqnarray*}
BT_aB^{-1}f(z)&=&e^{-z^2/2}\inc f(u)e^{-\frac12\overline u^2-a^2-2a\overline u+\frac12(z+a+\overline u)^2}\,d\lambda(u)\\
&=&\inc f(u)e^{-\frac12a^2-a\overline u+az+\overline uz}\,d\lambda(u)\\
&=&e^{-\frac12a^2+az}\inc f(u)e^{(z-a)\overline u}\,d\lambda(u)\\
&=&e^{-\frac12a^2+az}f(z-a)=W_af(z).
\end{eqnarray*}
This completes the proof of the theorem.
\end{proof}

\begin{thm}
For any real $b$ we have $BM_bB^{-1}=W_{-\pi bi}$.
\label{9}
\end{thm}

\begin{proof}
For any polynomial $f$ in $F^2$ we have
$$M_bB^{-1}f(x)=\ccc e^{2\pi bix}\inc f(z)e^{-(x-\overline z)^2+\frac{\overline z^2}2}\,d\lambda(z),$$
and
\begin{eqnarray*}
BM_bB^{-1}f(z)&=&\cccc e^{\frac{z^2}2}\!\!\inr e^{-(x-z)^2+2\pi bix}\,dx\!\!\inc f(u)e^{-(x-\overline u)^2
+\frac{\overline u^2}2}\,d\lambda(u)\\
&=&\cccc e^{\frac{z^2}2}\inc f(u)e^{\frac12\overline u^2}\,d\lambda(u)\inr e^{-(x-z)^2-(x-\overline u)^2+2\pi bix}\,dx\\
&=&\cccc e^{-\frac{z^2}2}\inc f(u)e^{-\frac{\overline u^2}2}\,d\lambda(u)\inr e^{-2x^2+2x(z+\overline u+\pi bi)}\,dx\\
&=&\cccc e^{-\frac{z^2}2}\!\!\inc f(u)e^{-\frac{\overline u^2}2+\frac12(z+\overline u+\pi bi)^2}d\lambda(u)
\!\!\inr e^{-2[x-(z+\overline u+\pi bi)/2]^2}dx\\
&=&e^{-\frac{z^2}2}\inc f(u)e^{-\frac{\overline u^2}2+\frac12(z+\overline u+\pi bi)^2}\,d\lambda(u)\\
&=&e^{-\frac12\pi^2b^2+\pi biz}\inc f(u)e^{(z+\pi bi)\overline u}\,d\lambda(u)\\
&=&e^{-\frac12\pi^2b^2+\pi biz}f(z+\pi bi)=W_{-\pi bi}f(z).
\end{eqnarray*}
The proves the desired result.
\end{proof}

\begin{cor}
Let $a$ and $b$ be any pair of real numbers. Then we have
$$B(M_bT_a)B^{-1}=e^{\pi abi}W_{a-\pi bi}.$$
Consequently, if $g\in L^2(\R,dx)$ and $f=Bg$. Then the Bargmann transform maps the function
$M_bT_ag$ to the function $e^{\pi abi}W_{a-\pi bi}f$.
\label{10}
\end{cor}

\begin{proof}
This follows from the two theorems above and the identity (2.22) on page 61 of \cite{Z2}.
\end{proof}

\section{Gabor frames}

A sequence $\{f_n\}$ in a Hilbert space $H$ is called a frame if there exists a positive constant $C$ such that
$$C^{-1}\|f\|^2\le\sum_{n=1}^\infty|\langle f,f_n\rangle|^2\le C\|f\|^2$$
for all $f\in H$. The following result contains the most important properties of frames in a Hilbert space.

\begin{thm}
A sequence $\{f_n\}$ in a Hilbert space $H$ is a frame if and only if the following conditions hold.
\begin{enumerate}
\item[(i)] For any $\{c_n\}\in l^2$ the series $\sum c_nf_n$ converges in the norm topology of $H$.
\item[(ii)] For any $f\in H$, there exists a sequence $\{c_n\}\in l^2$ such that $f=\sum c_nf_n$.
\end{enumerate}
\label{11}
\end{thm}

The result above is well known and is the foundation for the theory of frames. See \cite{C} for this and 
other properties of general frames in Hilbert spaces.

Let $\{a_n\}$ and $\{b_n\}$ be two sequences of real numbers, each consisting of distinct values, and let 
$g\in L^2(R)$. Write $g_n=M_{b_n}T_{a_n}g$ for $n\ge1$. We say that $\{g_n\}$ is a Gabor frame for 
$L^2(\R)$ if there exists a positive constant $C$ such that
$$C^{-1}\|f\|^2\le\sum_{n=1}^\infty|\langle f,g_n\rangle|^2\le C\|f\|^2$$
for all $f\in L^2(\R)$. In this case, $g$ is called the window function of the Gabor frame $\{g_n\}$.

By Theorem~\ref{11} above, if $\{g_n\}$ is a Gabor frame, then every function $f\in L^2(\R)$ can be 
represented in the form
$$f=\sum_{n=1}^\infty c_ng_n,$$
where $\{c_n\}\in l^2$ and the series converges in the norm topology of $L^2(\R)$. Conversely, if $\{g_n\}$ 
is a Gabor frame and $\{c_n\}\in l^2$, then the series above converges to some $f\in L^2(\R)$ in the norm 
topology of $L^2(\R)$. The representation is generally not unique though.

There are two classical examples of window functions in Gabor analysis.

Let $g(x)$ be the characteristic function of the unit interval $[0,1)$ and let $\{z_n=a_n+ib_n\}$ denote any fixed
arrangement of the square lattice $\Z^2$ into a sequence. Then $\{g_n\}$ is not only a Gabor frame, it is also an 
orthonormal basis. To see this, it is more transparent to use two indices:
$$g_{nm}(x)=\chi_{[n,n+1)}(x)e^{2m\pi ix},\qquad (n,m)\in\Z^2.$$
Now it is clear that
$$f(x)=\sum_{n\in\Z}f(x)\chi_{[n,n+1)}(x)=\sum_{n\in\Z}f_n(x),$$
where the functions $f_n(x)=f(x)\chi_{[n,n+1)}(x)$ are mutually orthogonal in $L^2(\R)$. On the other hand, for each 
$n\in\Z$, the function $f_n(x)$ is compactly support on $[n,n+1)$, so it can be expanded into a Fourier series:
$$f_n(x)=\sum_{m\in\Z}c_{nm}e^{2m\pi ix},\qquad n\le x<n+1,$$
whose terms are mutually orthogonal over $[n,n+1)$. Therefore, we have
$$f(x)=\sum_{(n,m)\in\Z^2}c_{nm}\chi_{[n,n+1)}(x)e^{2m\pi ix}=\sum_{(n,m)\in\Z^2}c_{nm}g_{nm}(x),$$
where $g_{nm}=M_mT_ng$ are mutually orthogonal in $L^2(\R)$.

Thus the characteristic function of the unit interval $[0,1)$ is a Gabor window, which is usually called a box window.

Another important Gabor window is the Gauss function $g(x)=e^{-x^2}$, which can be thought of as a smooth analog 
of the box window. However, it is not a trivial matter to see that the Gauss function is a Gabor window. It will become 
clear once we establish the connection with the Fock space. Furthermore, using the theory of Fock spaces, we will be 
able to know exactly which translation and modulation sequences give rise to Gabor frames for the Gauss window.

First observe that the following result is a consequence of Corollary~\ref{10}.

\begin{cor}
Let $\{a_n\}$ and $\{b_n\}$ be two sequences of real numbers, each consisting of distinct points in $\R$.
Let $g(x)=e^{-x^2}$ be the Gauss window function. Then the system $\{g_n=M_{b_n}T_{a_n}g\}$ is a
Gabor frame for $L^2(\R)$ if and only if the sequence $\{k_{z_n}\}$ is a frame for $F^2$, where
$z_n=a_n-\pi bi_n$ and $k_{z_n}$ is the normalized reproducing kernel of $F^2$ at $z_n$.
\label{12}
\end{cor}

The are many fundamental questions that Gabor analysis tries to address. Among them we mention the
following two.
\begin{enumerate}
\item[(1)] Characterize all window functions. Recall that we say a function $g\in L^2(\R)$ is a Gabor window if 
there exist two sequences $\{a_n\}$ and $\{b_n\}$ such that $\{g_n=M_{b_n}T_{a_n}g\}$ is a Gabor frame.
\item[(2)] Given a Gabor window $g$, characterize all sequences $\{a_n\}$ and $\{b_n\}$ such that 
$\{M_{b_n}T_{a_n}g\}$ is a Gabor frame.
\end{enumerate}

The most popular Gabor frames are constructed using rectangular lattices $a\Z\times b\Z$. Such Gabor frames will 
be called regular. Slightly more general are lattices based on congruent parallelograms: $\omega_1\Z+\omega_2\Z$, 
where $\omega_1$ and $\omega_2$ are two nonzero complex numbers which are linearly independent in $\R^2$. 
Gabor frames based on such lattices will also be called regular.

Semi-regular lattices are of the form $\{a_n\}\omega_1+\{b_m\}\omega_2$, where $\{a_n\}$ and $\{b_m\}$ are two 
sequences of distinct real numbers, and $\omega_1$ and $\omega_2$ are two complex numbers that are linearly 
independent in $\R^2$. These are lattices based on parallelograms of different sizes. The resulting Gabor frames 
will be called semi-regular.

Irregular Gabor frames are constructed using an arbitrary sequence $z_n=(a_n,b_n)$ in $\R^2$.

To understand the two questions raised earlier about Gabor frames, we need the notions of atomic decomposition 
and sampling sequences for $F^2$.

Let $\{z_n\}$ denote a sequence of (distinct) points in $\C$. We say that atomic decomposition for $F^2$
holds on the sequence $\{z_n\}$ if
\begin{enumerate}
\item[(1)] For any sequence $\{c_n\}\in l^2$ the series
$$\sum_{n=1}^\infty c_nk_{z_n}(z)$$
converges in the norm topology of $F^2$.
\item[(2)] For any function $f\in F^2$ there exists a sequence $\{c_n\}\in l^2$ such that
$$f(z)=\sum_{n=1}^\infty c_nk_{z_n}.$$
\end{enumerate}

We say that the sequence $\{z_n\}$ is sampling for $F^2$ if there exists a positive constant $C$ such that
$$C^{-1}\|f\|^2\le\sum_{n=1}^\infty|f(z_n)|^2e^{-|z_n|^2}\le C\|f\|^2$$
for all $f\in F^2$. See \cite{Z2} for basic information about atomic decomposition and sampling in $F^2$.

\begin{thm}
Given a sequence $\{z_n=a_n-\pi b_ni\}$ of distinct points in $\C$, the following three conditions are equivalent.
\begin{enumerate}
\item[(1)] The system $\{g_n=M_{b_n}T_{a_n}g\}$ is a Gabor frame for $L^2(\R)$, where $g$
is the Gauss window.
\item[(2)] The sequence $\{z_n\}$ is sampling for $F^2$.
\item[(3)] Atomic decomposition for $F^2$ holds on $F^2$.
\end{enumerate}
\label{13}
\end{thm}

\begin{proof}
Recall from Corollary \ref{12}  that $\{g_n\}$ is a Gabor frame for $L^2(\R)$ iff $\{k_{z_n}\}$ is a frame for $F^2$. 
Since $\langle f,k_z\rangle=f(z)e^{-\frac12|z|^2}$, we see that $\{k_{z_n}\}$ is a frame for $F^2$ if and only
if $\{z_n\}$ is a sampling sequence for $F^2$. It is well known that $\{z_n\}$ is a sampling sequence if and only
if atomic decomposition holds on $\{z_n\}$; see \cite{Z2} for example.
\end{proof}

More generally, we have the following.

\begin{thm}
Suppose $\{a_n\}$ and $\{b_n\}$ are sequences of real numbers and $g\in L^2(\R)$. Then
$\{M_{b_n}T_{a_n}g\}$ is a Gabor frame for $L^2(\R)$ if and only if $\{W_{z_n}f\}$ is a
frame in $F^2$, where $z_n=a_n-\pi bi$, $W_{z_n}$ are the Weyl operators, and $f=Bg$.
\label{14}
\end{thm}

\begin{proof}
This follows from Corollary~\ref{10} again.
\end{proof}

Thus the study of Gabor frames for $L^2(\R)$ is equivalent to the study of frames of the form $\{W_{z_n}f\}$ 
in $F^2$. In particular, Gabor frames with the Gauss window correspond to frames $\{k_{z_n}\}$ in $F^2$ which 
in turn correspond to sampling sequences $\{z_n\}$ for $F^2$.

Sampling sequences have been completely characterized by Seip and Wallst\'en. To describe 
their results, we need to introduce a certain notion of density for sequences in the complex plane. Thus for a 
sequence $Z=\{z_n\}$ of distinct points in the complex plane, we define
$$D^+(Z)=\limsup_{R\to\infty}\sup_{z\in\C}\frac{|Z\cap B(z,R)|}{\pi R^2},$$
and
$$D^-(Z)=\liminf_{R\to\infty}\inf_{z\in\C}\frac{|Z\cap B(z,R)|}{\pi R^2},$$
where 
$$B(z,R)=\{w\in\C:|w-z|<R\}$$
and $|Z\cap B(z,R)|$ denotes the cardinality of $Z\cap B(z,R)$. These are called the Beurling upper and lower 
densities of $Z$, respectively. The following characterization of sampling sequences for $F^2$ can be found 
in \cite{Seip2, SW}.

\begin{thm}
A sequence $Z$ of distinct points in $\C$ is sampling for $F^2$ if and only if the following two conditions
are satisfied.
\begin{enumerate}
\item[(a)] $Z$ is the union of finitely many subsequences each of which is separated in the Euclidean metric.
\item[(b)] $Z$ contains a subsequence $Z'$ such that $D^-(Z')>1/\pi$.
\end{enumerate}
\label{15}
\end{thm}

As a consequence of Theorems \ref{13} and \ref{15} we obtain the following complete characterization of 
Gabor frames in $L^2(\R)$ associated with the Gauss window.

\begin{thm}
Let $Z=\{z_n=a_n-\pi ib_n:n\ge0\}$ be a sequence of points in $\C=\R^2$ and let $g(x)$ be the Gauss window. 
Then $\{M_{b_n}T_{a_n}g\}$ is a Gabor frame in $L^2(\R)$ if and only if the following two conditions are satisfied:
\begin{enumerate}
\item[(a)] $Z$ is the union of finitely many subsequences each of which is separated in the Euclidean metric.
\item[(b)] $Z$ contains a subsequence $Z'$ such that $D^-(Z')>1/\pi$.
\end{enumerate}
\label{16}
\end{thm}

\begin{cor}
Let $g$ be the Gauss window and $Z=\{z_n=a_n-\pi ib_n:n\ge0\}$ be a sequence in $\C$ that is separated 
in the Euclidean metric. Then $\{M_{b_n}T_{a_n}g\}$ is a Gabor frame for $L^2(\R)$ if and only if $D^-(Z)>1/\pi$.
\label{17}
\end{cor}

The most classical case is when the sequence $Z$ is a rectangular lattice in $\C$, which is clearly separated in 
the Euclidean metric. Thus we have the following.

\begin{cor}
Let $g$ be the Gauss window and $a,b$ be positive constants. Then $\{M_{mb}T_{na}g:m,n\in\Z\}$ is a Gabor 
frame if and only if $ab<1$.
\label{18}
\end{cor}

\begin{proof}
When $R$ is very large, it is easy to see that $|Z\cap B(z,R)|$ is roughly $\pi R^2/(\pi ab)$. It follows that 
$$D^-(Z)=D^+(Z)=1/(\pi ab).$$
Therefore, according to Corollary~\ref{17}, $\{M_{mb}T_{na}g:m,n\in\Z\}$ is a Gabor frame if and only if 
$1/(\pi ab)>1/\pi$, or $ab<1$.
\end{proof}

It is interesting to observe that, for the Gauss window, whether or not a rectangular lattice $Z=\{na+mib:n,m\in\Z\}$ 
generates a Gabor frame only depends on the product $ab$. This was already a well-known fact (and a conjecture
for a long time) before Seip and Wallst\'en gave a complete characterization of sampling sequences for the Fock space. 
The Bargmann transform allows us to obtain a complete characterization of (regular AND irregular) Gabor frames 
corresponding to the Gauss window.

In addition to the Gauss window, we also mentioned the box window $g(x)=\chi_{[0,1)}(x)$ earlier. In this case,
the Gabor frame
$$g_{mn}(x)=\chi_{[n,n+1)}(x)e^{2m\pi ix},\qquad (m,n)\in\Z^2,$$
is actually an orthogonal basis for $L^2(\R)$. Via the Bargmann transform, we have
$$f(z)=Bg(z)=c\int_0^1e^{2xz-x^2-(z^2/2)}\,dx,$$
and the corresponding orthogonal basis for $F^2$ is given by
$$f_{mn}(z)=W_{z_{mn}}f(z)=f(z-z_{mn})k_{z_{mn}}(z)=f(z-z_{mn})e^{z\overline z_{mn}-\frac{|z_{mn}|^2}2},$$
where $z_{mn}=n-m\pi i$ for $(m,n)\in\Z^2$. Although the box window and the associated orthogonal basis
$\{g_{mn}\}$ are very natural in Gabor analysis, their counterparts in the Fock space have not yet been studied.

\section{The canonical commutation relation}

There are two unbounded operators that are very important in the study of $L^2(\R)$, namely, the operator of 
multiplication by $x$ and the operator of differentiation. Thus for this section we write
$$Tf(x)=xf(x),\qquad Sf(x)=f'(x),\qquad f\in L^2(\R).$$
It is clear that both of them are unbounded but densely defined. We will identify the operators on
$F^2$ that correspond to these two operators under the Bargmann transform.

\begin{thm}
For $f\in F^2$ we have
$$BTB^{-1}f(z)=\frac12\left[zf(z)+f'(z)\right].$$
\label{19}
\end{thm}

\begin{proof}
Let $c=(2/\pi)^{1/4}$ and $f\in F^2$ (a polynomial for the sake of using Fubini's theorem). We have
$$TB^{-1}f(x)=cx\inc f(w)e^{2x\overline w-x^2-\frac{\overline w^2}2}\,d\lambda(w),$$
and so
\begin{eqnarray*}
BTB^{-1}f(z)&=&c^2\inr xe^{2xz-x^2-(z^2/2)}\,dx\inc f(w)e^{2x\overline w-x^2-(\overline w^2/2)}\,d\lambda(w)\\
&=&c^2e^{-z^2/2}\inc f(w)e^{-\overline w^2/2}d\lambda(w)\inr xe^{-2x^2+2x(z+\overline w)}\,dx\\
&=&\frac{c^2}2e^{-\frac{z^2}2}\inc f(w)e^{-\frac{\overline w^2}2}\,d\lambda(w)\inr xe^{-x^2+2x(z+\overline w)/\sqrt2}\,dx\\
&=&\frac{c^2}2\inc f(w)e^{z\overline w}\,d\lambda(w)\inr xe^{-[x-(z+\overline w)/\sqrt2]^2}\,dx\\
&=&\frac{c^2}2\inc f(w)e^{z\overline w}\,d\lambda(w)\inr\left(x+\frac{z+\overline w}{\sqrt2}\right)e^{-x^2}\,dx\\
&=&\frac12\inc f(w)(z+\overline w)e^{z\overline w}\,d\lambda(w)\\
&=&\frac12\left[zf(z)+\frac d{dz}\inc f(w)e^{z\overline w}\,d\lambda(w)\right]\\
&=&\frac12[zf(z)+f'(z)].
\end{eqnarray*}
This proves the desired result.
\end{proof}

\begin{thm}
For $f\in F^2$ we have
$$BSB^{-1}f(z)=f'(z)-zf(z).$$
\label{20}
\end{thm}

\begin{proof}
Let $f\in F^2$ and $c=(2/\pi)^{1/4}$ again. We have
\begin{eqnarray*}
SB^{-1}f(x)&=&c\inc f(w)(2\overline w-2x)e^{2x\overline w-x^2-(\overline w^2/2)}\,d\lambda(w)\\
&=&2c\inc\overline wf(w)e^{2x\overline w-x^2-(\overline w^2/2)}\,d\lambda(w)-2xB^{-1}f(x).
\end{eqnarray*}
Let $g(x)$ denote the function defined by $2c$ times the integral above. Then
\begin{eqnarray*}
Bg(z)&=&2c^2\inr e^{2xz-x^2-(z^2/2)}\,dx\inc\overline wf(w)e^{2x\overline w-x^2-(\overline w^2/2)}\,d\lambda(w)\\
&=&2c^2e^{-\frac{z^2}2}\inc\overline wf(w)e^{-\frac{\overline w^2}2}\,d\lambda(w)\inr e^{-2x^2+2x(z+\overline w)}\,dx\\
&=&\sqrt2c^2e^{-\frac{z^2}2}\inc\overline wf(w)e^{-\frac{\overline w^2}2}\,d\lambda(w)
\inr e^{-x^2+2x(z+\overline w)/\sqrt2}\,dx\\
&=&\sqrt2c^2\inc\overline wf(w)e^{z\overline w}\,d\lambda(w)\inr e^{-[x-(z+\overline w)/\sqrt2]^2}\,dx\\
&=&\sqrt{2\pi}c^2\inc\overline wf(w)e^{z\overline w}\,d\lambda(w)\\
&=&2f'(z).
\end{eqnarray*}
It follows that
$$BSB^{-1}f(z)=2f'(z)-(zf(z)+f'(z))=f'(z)-zf(z).$$
This proves the desired result.
\end{proof}

Let $A_1=BTB^{-1}$ and $A_2=BSB^{-1}$. Thus
$$A_1f(z)=\frac12[f'(z)+zf(z)],\qquad A_2f(z)=f'(z)-zf(z).$$
It is easy to verify that
$$(A_2A_1-A_1A_2)f=f,\qquad f\in F^2.$$
This gives the classical commutation relation as follows.

\begin{cor}
On the space $L^2(\R)$ we have $[S,T]=I$, and on the space $F^2$ we have $[A_2,A_1]=I$. Here $I$
denotes the identity operator on the respective spaces.
\label{21}
\end{cor}

Note that if we simply define
$$Mf(z)=zf(z),\qquad Df(z)=f'(z),\qquad f\in F^2,$$
then we also have 
$$DM-MD=[D,M]=I.$$
Thus the following problem becomes interesting in functional analysis: characterize all operator pairs (potentially 
unbounded) on a Hilbert space such that their commutator is equal to the identity operator.

\section{Uncertainty principles}

With our normalization of the Fourier transform, namely,
$$\widehat f(x)=\frac1{\sqrt\pi}\inr f(t)e^{2ixt}\,dt,$$
the classical uncertainty principle in Fourier analysis takes the following form:
$$\frac14\|f\|^2\le\|(x-a)f\|\|(x-b)\widehat f\|,$$
where the norm is taken in $L^2(\R)$, and $a$ and $b$ are arbitrary real numbers. See \cite{F2, G1}
for this and for exactly when equality holds.

Thus the corresponding result for the Fock space is
$$\frac14\|f\|^2\le\|(T-a)f\|\|(T-b)\widehat f\|,$$
where the norm is taken in the Fock space, the operator $T$ (unitarily equivalent to
the operator of multiplication by $x$ on $L^2(\R)$) is given by
$$Tf(z)=\frac12(f'(z)+zf(z)),$$
and $a$ and $b$ are arbitrary real numbers. Recall from Theorem~\ref{3} that in the context of the Fock space 
we have $\widehat f(z)=f(iz)$. Therefore, the corresponding uncertainty principle for the Fock space is
$$\|f'(z)+zf(z)-af(z)\|\|if'(iz)+zf(iz)-bf(iz)\|\ge\|f\|^2,$$
where $a$ and $b$ are arbitrary real numbers. Rewrite this as
$$\|f'(z)+zf(z)-af(z)\|\|f'(iz)-izf(iz)+bif(iz)\|\ge\|f\|^2.$$
Since the norm in $F^2$ is rotation invariant, we can rewrite the above as
$$\|f'+zf-af\|\|f'-zf+bif\|\ge\|f\|^2.$$
Changing $b$ to $-b$, we obtain the following version of the uncertainty principle for the Fock space.

\begin{thm}
Suppose $a$ and $b$ are real constants. Then
$$\|f'+zf-af\|\|f'-zf-bif\|\ge\|f\|^2$$
for all $f\in F^2$ (with the understanding that the left-hand side may be infinite). Moreover, equality holds if and only if
$$f(z)=C\exp\left(\frac{c-1}{2(c+1)}z^2+\frac{a+ibc}{c+1}z\right),$$
where $C$ is any complex constant and $c$ is any positive constant.
\label{22}
\end{thm}

\begin{proof}
We only need to figure out exactly when equality occurs. This can be done with the help of the Bargmann
transform and the known condition about when equality occurs in the classical version of the uncertainty
principle in Fourier analysis. However, it is actually easier to do it directly in the context of the Fock space.

More specifically, we consider the following two self-adjoint operators on $F^2$:
$$S_1f=f'+zf,\qquad S_2f=i(f'-zf),\qquad f\in F^2.$$
It is easy to check that $S_1$ and $S_2$ satisfy the commutation relation
$$[S_1,S_2]=S_1S_2-S_2S_1=-2iI.$$
By the well-known functional analysis result on which the uncertainty principle is usually based (see 
\cite{F2, G1} for example), we have
$$\|(S_1-a)f\|\|(S_2-b)f\|\ge\frac12|\langle[S_1,S_2]f,f\rangle|=\|f\|^2$$
for all $f\in F^2$ and all real constants $a$ and $b$. Moreover, equality holds if and only if $(S_1-a)f$ and
$(S_2-b)f$ are purely imaginary scalar multiples of one another. It follows that
$$\|f'+zf-af\|\|f'-zf-ibf\|\ge\|f\|^2,$$
with equality if and only if
\begin{equation}
f'+zf-af=ic[i(f'-zf)+bf],
\label{eq1}
\end{equation}
or 
\begin{equation}
i(f'-zf)+bf=ic[f'+zf-af],
\label{eq2}
\end{equation}
where $c$ is a real constant. In the first case, we can rewrite the equality condition as
$$(1+c)f'+[(1-c)z-(a+ibc)]f=0.$$
If $c=-1$, the only solution is $f=0$, which can be written in the form (\ref{eq3}) below with $C=0$ 
and arbitrary positive $c$. If $c\not=-1$, then it is elementary to solve the first order linear ODE to get
\begin{equation}
f(z)=C\exp\left(\frac{c-1}{2(c+1)}z^2+\frac{a+ibc}{c+1}z\right),
\label{eq3}
\end{equation}
where $C$ is any complex constant. It is well known that every function $f\in F^2$ must satisfy 
the growth condition
\begin{equation}
\lim_{z\to\infty}f(z)e^{-|z|^2/2}=0.
\label{eq4}
\end{equation}
See page 38 of \cite{Z2} for example. Therefore, a necessary condition for the function in (\ref{eq3})
to be in $F^2$ is $C=0$ or $|c-1|\le|c+1|$. Since $c$ is real, we must have either $C=0$ or $c\ge0$. 
When $c=0$, the function in (\ref{eq3}) becomes
$$f(z)=C\exp\left(-\frac12z^2+az\right),$$
which together with (\ref{eq4}) forces $C=0$. The case of (\ref{eq2}) is dealt with in a similar manner. 
This completes the proof of the theorem.
\end{proof}

The result above was published in Chinese in \cite{CZ}, where several other versions of the uncertainty
principle were also obtained. We included some details here for the convenience of those readers who
are not familiar with Chinese.

\section{The Hilbert transform}

The Hilbert transform is the operator on $L^2(\R)$ defined by
$$Hf(x)=\frac1\pi\inr\frac{f(t)\,dt}{t-x},$$
where the improper integral is taken in the sense of ``principal value". This is a typical ``singular integral
operator".

The Hilbert transform is one of the most studied objects in harmonic analysis. It is well known that $H$ is
a bounded linear operator on $L^p(\R)$ for every $1<p<\infty$, and it is actually a unitary operator on
$L^2(\R)$. See \cite{F2} for example.

In order to identify the corresponding operator on the Fock space, we need the entire function
$$A(z)=\int_0^z e^{u^2}\,du,\qquad z\in\C,$$
which is the antiderivative of $e^{z^2}$ satisfying $A(0)=0$. We will also need the following

\begin{lem}
We have
$$\inr e^{-(x-z)^2}\,dx=\sqrt\pi$$
for every complex number $z$.
\label{23}
\end{lem}

\begin{proof}
Let $I(z)$ denote the integral in question. Then $I(z)$ is clearly an entire function. Since
$$I'(z)=2\inr(x-z)e^{-(x-z)^2}\,dx=\left.e^{-(x-z)^2}\right|^{+\infty}_{-\infty}=0,$$
we must have
$$I(z)=I(0)=\inr e^{-x^2}\,dx=\sqrt\pi$$
for all $z\in\C$.
\end{proof}

Given a function $f\in F^2$ (we may start out with a polynomial in order to justify the use of Fubini's theorem), we have
\begin{eqnarray*}
HB^{-1}f(x)&=&\frac1\pi\inr\frac{B^{-1}f(t)\,dt}{t-x}\\
&=&\frac{\ccc}{\pi}\inr\frac{dt}{t-x}\inc f(z)e^{2t\overline z-t^2-(\overline z^2/2)}\,d\lambda(z)\\
&=&\frac{\ccc}{\pi}\inc f(z)e^{-\frac{\overline z^2}2}\,d\lambda(z)\inr\frac{e^{2t\overline z-t^2}}{t-x}\,dt\\
&=&\frac{\ccc}{\pi}\inc f(z)e^{2x\overline z-x^2-\frac{\overline z^2}2}\,d\lambda(z)\inr\frac{e^{-t^2+2t(\overline z-x)}}t\,dt.
\end{eqnarray*}
Let us consider the entire function
$$h(u)=\inr\frac{e^{-t^2+2tu}}{t}\,dt.$$
We can rewrite this PV-integral in the form of an ordinary integral as follows:
$$h(u)=\inr\frac{e^{-t^2}(e^{2tu}-1)}t\,dt.$$
The singularity at $t=0$ and the singularity at infinity are both gone. Thus we can differentiate inside
the integral sign to get
$$h'(u)=2\inr e^{-t^2+2tu}\,dt=2e^{u^2}\inr e^{-(t-u)^2}\,dt=2\sqrt\pi e^{u^2}.$$
Since $h(0)=0$, we must have $h(u)=2\sqrt\pi A(u)$. Thus
$$HB^{-1}f(x)=\frac{2c}{\sqrt\pi}\inc f(w)e^{2x\overline w-x^2-\frac{\overline w^2}2}A(\overline w-x)\,d\lambda(w).$$
Therefore,
\begin{eqnarray*}
&&BHB^{-1}f(z)\\
&=&\frac{2\sqrt2}\pi\inr e^{2xz-x^2-\frac{z^2}2}\,dx
\inc f(w)e^{2x\overline w-x^2-\frac{\overline w^2}2}A(\overline w-x)\,d\lambda(w)\\
&=&\frac{2\sqrt2}\pi e^{-\frac{z^2}2}\inc f(w)e^{-\frac{\overline w^2}2}\,d\lambda(w)
\inr e^{-2x^2+2x(z+\overline w)}A(\overline w-x)\,dx.
\end{eqnarray*}

Fix $w$ and consider the entire function
$$J(z)=\inr e^{-2x^2+2x(\overline w+z)}A(\overline w-x)\,dx.$$
We have
\begin{eqnarray*}
J'(z)&=&2\inr xe^{-2x^2+2x(\overline w+z)}A(\overline w-x)\,dx\\
&=&\inr[2x-(\overline w+z)+(\overline w+z)]e^{-2x^2+2x(\overline w+z)}A(\overline w-x)\,dx\\
&=&-\frac12\inr A(\overline w-x)de^{-2x^2+2x(\overline w+z)}+\\
&&\ \ +(\overline w+z)\inr e^{-2x^2+2x(\overline w+z)}A(\overline w-x)\,dx\\
&=&-\frac12\inr e^{-2x^2+2x(\overline w+z)}e^{(\overline w-x)^2}\,dx+(\overline w+z)J(z)\\
&=&-\frac12e^{\overline w^2+z^2}\inr e^{-(x-z)^2}\,dx+(\overline w+z)J(z)\\
&=&-\frac{\sqrt\pi}2e^{\overline w^2+z^2}+(\overline w+z)J(z).
\end{eqnarray*}
We can rewrite this in the following form:
$$\frac{d}{dz}\left[J(z)e^{-\frac12(\overline w+z)^2}\right]=-\frac{\sqrt\pi}2e^{\frac12(\overline w-z)^2}.$$
It follows that
$$J(z)e^{-\frac12(\overline w+z)^2}=-\sqrt{\frac\pi2}A\left(\frac{z-\overline w}{\sqrt2}\right)+C(w).$$
We are going to show that $C(w)=0$. To this end, let $z=-\overline w$ in the identity above. We obtain
$$C(w)=J(-\overline w)+\sqrt{\frac\pi2}A(-\sqrt2\,\overline w).$$
Let 
$$F(\overline w)=J(-\overline w)=\inr e^{-2x^2}A(\overline w-x)\,dx,$$
or
$$F(u)=\inr e^{-2x^2}A(u-x)\,dx,\qquad u\in\C.$$
We have
\begin{eqnarray*}
F'(u)&=&\inr e^{-2x^2}e^{(u-x)^2}\,dx\\
&=&e^{2u^2}\inr e^{-(x+u)^2}\,dx\\
&=&\sqrt\pi e^{(\sqrt2 u)^2}.
\end{eqnarray*}
It follows that
$$F(u)=\sqrt{\frac\pi2}A(\sqrt2\,u)+C.$$
Since $A(0)=0$ and $A(u)$ is odd (because $e^{u^2}$ is even), we have
$$F(0)=\inr e^{-2x^2}A(-x)\,dx=0.$$
This shows that $C=0$, or
$$F(u)=\sqrt{\frac\pi2}A(\sqrt2\,u).$$
Going back to the formula for $C(w)$, we obtain
$$C(w)=\sqrt{\frac\pi2}A(\sqrt2\,\overline w)+\sqrt{\frac\pi2}A(-\sqrt2\,\overline w)=0,$$
because $A(u)$ is odd again. Therefore,
$$J(z)=-\sqrt{\frac\pi2}A\left(\frac{z-\overline w}{\sqrt2}\right)e^{\frac12(\overline w+z)^2},$$
and
\begin{eqnarray*}
BHB^{-1}f(z)&=&-\frac2{\sqrt\pi}e^{-\frac{z^2}2}\inc f(w)e^{-\frac12\overline w^2+\frac12(\overline w+z)^2}\,
A\left(\frac{z-\overline w}{\sqrt2}\right)\,d\lambda(w)\\
&=&-\frac2{\sqrt\pi}\inc f(w)e^{z\overline w}A\left(\frac{z-\overline w}{\sqrt2}\right)\,d\lambda(w).
\end{eqnarray*}
We summarize the result of this analysis as the following theorem.

\begin{thm}
Suppose $A(z)$ is the anti-derivative of $e^{z^2}$ with $A(0)=0$ and $T=BHB^{-1}$. Then
$$Tf(z)=-\frac2{\sqrt\pi}\inc f(w)e^{z\overline w}A\left(\frac{z-\overline w}{\sqrt2}\right)\,d\lambda(w)$$
for $f\in F^2$ and $z\in\C$.
\label{24}
\end{thm}

This appears to be a very interesting integral operator on the Fock space. Note that we clearly have
$$T(1)(z)=-\frac2{\sqrt\pi}\,A\left(\frac z{\sqrt2}\right),$$
which must be a function in $F^2$. The following calculation gives an alternate proof that this function is 
indeed in $F^2$.

\begin{lem}
The function
$$f(z)=A\left(\frac z{\sqrt2}\right)$$
belongs to the Fock space $F^2$.
\label{25}
\end{lem}

\begin{proof}
Since
$$e^{z^2}=\sum_{n=0}^\infty\frac1{n!}\,z^{2n},$$
we have
$$A(z)=\sum_{n=0}^\infty\frac1{(2n+1)n!}\,z^{2n+1},$$
and so
\begin{eqnarray*}
f(z)&=&\frac1{\sqrt2}\sum_{n=0}^\infty\frac1{(2n+1)2^nn!}\,z^{2n+1}\\
&=&\frac1{\sqrt2}\sum_{n=0}^\infty\frac{\sqrt{(2n+1)!}}{(2n+1)2^nn!}\,e_{n+1}(z),
\end{eqnarray*}
where $\{e_n\}$ is the standard monomial orthonormal basis for $F^2$. It follows that
$$\|f\|^2=\frac12\sum_{n=0}^\infty\frac{(2n+1)!}{(2n+1)^24^n(n!)^2}.$$
By Stirling's formula, it is easy to check that
$$\|f\|^2\sim\sum_{n=1}^\infty\frac1{n^{3/2}}<\infty.$$
This proves the desired result.
\end{proof}

A natural problem here is to study the spectral properties of the integral operator $T$ above (or
equivalently, the Hilbert transform as an operator on $L^2(\R)$):
fixed-points, eigenvalues, spectrum, invariant subspaces, etc. Not much appears to be
known, which is in sharp contrast to the case of the Fourier transform. Recall that Corollary~\ref{6} 
gives a complete spectral picture for the Fourier transform $F$ as an operator on $L^2(\R)$.

Motivated by Theorem \ref{24}, we consider more general ``singular integral operators" on $F^2$ of the form
$$S_\varphi f(z)=\inc f(w)e^{z\overline w}\varphi(z-\overline w)\,d\lambda(w),$$
where $\varphi$ is any function in $F^2$. The most fundamental problem here is to characterize those 
$\varphi\in F^2$ such that $S_\varphi$ is bounded on $F^2$.

To show that the problem is interesting and non-trivial, we present several examples in the rest of this section.

If $\varphi=1$, it follows from the reproducing property of the kernel function $e^{z\overline w}$ that $S_\varphi$ 
is the identity operator.

If $\varphi(z)=z$, it is easy to verify that 
$$S_\varphi f(z)=zf(z)-f'(z),$$ 
which shows that $S_\varphi$ is unbounded on $F^2$; see \cite{CZ2}. This may appear discouraging, as the 
function $\varphi(z)=z$ appears to be as nice as it can be (except constant functions) in $F^2$.
But we will see that there are many other nice functions $\varphi$ that induce bounded operators $S_\varphi$.

First consider functions of the form $\varphi(z)=e^{z\overline a}$. An easy calculation in \cite{Z3} shows that
$$S_\varphi f(z)=e^{(\overline a-a)z+\frac{|a|^2}2}W_{\overline a}f(z),$$
where $W_z$ are the Weyl operators (which are unitary on $F^2$). Thus $S_\varphi$ is bounded if and only if 
$a$ is real, because the only point-wise multipliers of the Fock space are constants.

Next, it was shown in \cite{Z3} again that the operator $S_\varphi$ induced by $\varphi(z)=e^{az^2}$, where
$0<a<1/2$, is bounded on $F^2$. Furthermore, the range for $a$ above is best possible.

Finally, it was observed in \cite{Z3} that the Berezin transform of $S_\varphi$ is given by
$$\langle S_\varphi k_z,k_z\rangle=\varphi(z-\overline z),\qquad z\in\C.$$
Therefore, a necessary condition for $S_\varphi$ to be bounded on $F^2$ is that the function
$\varphi$ be bounded on the imaginary axis. It would be nice to know how far away is the condition from
being sufficient as well.

\section{Pseudo-differential operators}

In this section we explain that, under the Bargmann transform and with mild assumptions on the symbol functions, 
Toeplitz operators on the Fock space are unitarily equivalent to pseudo-differential operators on $L^2(\R)$.

Recall that if $\varphi=\varphi(z)$ is a symbol function on the complex plane, the Toeplitz operator
$T_\varphi$ on $F^2$ is defined by $T_\varphi f=P(\varphi f)$, where 
$$P:L^2(\C,d\lambda)\to F^2$$
is the orthogonal projection. We assume that $\varphi$ is good enough so that the operator $T_\varphi$ is
at least densely defined on $F^2$.

To be consistent with the theory of pseudo-differential operators, in this section we use $X$ and $D$ to denote 
the following operators on $L^2(\R)$:
$$Xf(x)=xf(x),\qquad Df(x)=\frac{f'(x)}{2i}.$$
Both $X$ and $D$ are self-adjoint and densely defined on $L^2(\R)$. We also consider the following
operators on $F^2$:
$$Z=X+iD,\qquad Z^*=X-iD.$$

There are several notions of pseudo-differential operators. We mention two of them here. First, if
$$\sigma=\sigma(z,\overline z)=\sum a_{mn}z^n\overline z^m$$
is a real-analytic polynomial on the complex plane, we define
$$\sigma(Z,Z^*)=\sum a_{mn}Z^nZ^{*m}$$
and call it the anti-Wick pseudo-differential operator with symbol $\sigma$.

\begin{thm}
Suppose $\sigma=\sigma(z,\overline z)$ is a real-analytic polynomial and $\varphi(z)=\sigma(\overline z,z)$.
Then $B\sigma(Z,Z^*)B^{-1}=T_\varphi$.
\label{26}
\end{thm}

This was proved in \cite{Z3}, so we omit the details here.

A more widely used notion of pseudo-differential operators is defined in terms of the Fourier transform.
More specifically, if $\sigma=\sigma(\zeta,x)$ is a symbol function on $\R^2=\C$ and if $f\in L^2(\R)$, we define
$$\sigma(D,X)f(x)=\frac1\pi\inr\inr\sigma\left(\zeta,\frac{x+y}2\right)e^{2i(x-y)\zeta}f(y)\,dy\,d\zeta.$$
We assume that $\sigma$ is good enough so that the operator $\sigma(D,X)$ is at least densely defined
on $L^2(\R)$. We call $\sigma(D,X)$ the Weyl pseudo-differential operator with symbol $\sigma$.

\begin{thm}
Suppose $\varphi=\varphi(z)$ is a (reasonably good) function on the complex plane. For $z=x+i\zeta$ we define
$$\sigma(z)=\sigma(\zeta,x)=\frac2\pi\inc\varphi(\overline w)e^{-2|z-w|^2}\,dA(w).$$
Then $B\sigma(D,X)B^{-1}=T_\varphi$.
\label{27}
\end{thm}

Again, this was proved in \cite{Z3}. Furthermore, the result was used in \cite{Z3} to study Toeplitz operators on 
$F^2$ with the help of the more mature theory of pseudo-differential operators on $L^2(\R)$.

\section{Further results and remarks}

For any $0<p\le\infty$ let $F^p$ denote the space of entire functions $f$ such that the function 
$f(z)e^{-\frac12|z|^2}$ belongs to $L^p(\C,dA)$. For $p<\infty$ and $f\in F^p$ we write
$$\|f\|^p_{F^p}=\frac{p}{2\pi}\inc\left|f(z)e^{-\frac12|z|^2}\right|^p\,dA(z).$$
For $f\in F^\infty$ we define
$$\|f\|_{F^\infty}=\sup_{z\in\C}|f(z)|e^{-\frac12|z|^2}.$$
These spaces $F^p$ are also called Fock spaces.

Let $f\in L^\infty(\R)$ and let $c=(2/\pi)^{1/4}$. We have
\begin{eqnarray*}
|Bf(z)|&\le&c\|f\|_\infty|e^{-z^2/2}|\inr\left|e^{2xz-x^2}\right|\,dx\\
&=&c\|f\|_\infty e^{-\frac12\re(z^2)+(\re z)^2}\inr e^{-(x-\re z)^2}\,dx\\
&=&c\sqrt\pi\|f\|_\infty e^{\frac12|z|^2}.
\end{eqnarray*}
Here we used the elementary identity
$$-\frac12\re(z^2)+(\re z)^2=\frac12|z|^2,$$
which can be verified easily by writing $z=u+iv$. Therefore, we have shown that
$$\|Bf\|_{F^\infty}\le c\sqrt\pi\|f\|_\infty$$
for all $f\in L^\infty(\R)$. In other words, the Bargmann transform $B$ is a bounded linear
operator from $L^\infty(\R)$ into the Fock space $F^\infty$. By complex interpolation (see
\cite{Z2}), we have proved the following.

\begin{thm}
For any $2\le p\le\infty$, the Bargmann transform maps $L^p(\R)$ boundedly into $F^p$.
\label{28}
\end{thm}

It would be nice to find out the mapping properties of the Bargmann transform on the spaces
$L^p(\R)$ for $0<p<2$. A more tractable range may be $1<p<2$ or $1\le p<2$.

Among other mapping properties of the Bargmann transform we mention the following two: if $S_0$ is the
Schwarz class, what is the image of $S_0$ in $F^2$ under the Bargmann transform? And what properties 
does $Bf$ have if $f$ is a function in $L^2(\R)$ with compact support?

One of the fundamental questions in Gabor analysis is to characterize all possible window functions.
Via the Bargmann transform we know that this is equivalent to the following problem: characterize all
functions $g\in F^2$ such that $\{W_{z_n}g\}$ is a Gabor frame for some $\{z_n\}\subset\C$. In terms
of ``atomic decomposition'', this is also equivalent to the following problem: characterize functions
$g\in F^2$ such that for some $\{z_n\}$ we have ``atomic decomposition'' for $F^2$: functions of the form
$$f=\sum_{n=1}^\infty c_nW_{z_n}g,\qquad\{c_n\}\in l^2,$$
represent exactly the space $F^2$.

Another well-known problem in time-frequency analysis is the so-called ``linear independence'' problem;
see \cite{H}. More specifically, given any function $g\in L^2(\R)$ and distinct points $z_k=a_k+ib_k$, 
$1\le k\le N$, are the functions
$$g_k(x)=M_{a_k}T_{b_k}g(x)=e^{2\pi b_kix}g(x-a_k),\qquad 1\le k\le N,$$
linearly independent in $L^2(\R)$? Via the Bargmann transform, this is equivalent to the following
problem for the Fock space: given any nonzero function $f\in F^2$ and distinct points $z_k\in\C$, $1\le k\le N$, 
are the functions
$$W_{z_k}f(z)=e^{z\overline z_k-\frac{|z_k|^2}2}f(z-z_k),\qquad 1\le k\le N,$$
always linearly independent in $F^2$? Equivalently, are the functions
$$f_k(z)=e^{z\overline z_k}f(z-z_k),\qquad 1\le k\le N,$$
always linearly independent in $F^2$?


\begin{thebibliography}{99}

\bibitem{ALS} G. Ascensi, Y. Lyubarskii, and K. Seip, Phase space distribution of Gabor expansions,
{\sl Appl. Comput. Harmon. Anal.} {\bf 26} (2009), 277-282.
\bibitem{Barg1} V. Bargmann, On a Hilbert space of analytic functions and an associated integral transform I,
{\sl Comm. Pure Appl. Math.} {\bf 14} (1961), 187-214.
\bibitem{Barg2} V. Bargmann, On a Hilbert space of analytic functions and an associated integral transform II,
{\sl Comm. Pure Appl. Math.} {\bf 20} (1967), 1-101.
\bibitem{BCI} W. Bauer, L. Coburn, and J. Isralowitz, Heat flow, BMO, and compactness of Toeplitz
operators, {\sl J. Funct. Anal.} {\bf 259} (2010), 57-78.
\bibitem{Berezin1} F. Berezin, Covariant and contra variant symbols of operators, {\sl Math. USSR-Izv.}
{\bf 6} (1972), 1117-1151.
\bibitem{Berezin2} F. Berezin, Quantization, {\sl Math. USSR-Izv.} {\bf 8} (1974), 1109-1163.
\bibitem{BC1} C. Berger and L. Coburn, Toeplitz operators and quantum mechanics, {\sl J. Funct. Anal.}
{\bf 68} (1986), 273-299.
\bibitem{BC2} C. Berger and L. Coburn, Toeplitz operators on the Segal-Bargmann space, {\sl Trans. Amer.
Math. Soc.} {\bf 301} (1987), 813-829.
\bibitem{BC3} C. Berger and L. Coburn, Heat flow and Berezin-Toeplitz estimates, {\sl Amer. J. Math.}
{\bf 116} (1994), 563-590.
\bibitem{CZ} Y. Chen and K. Zhu, Uncertainty principles for the Fock space, {\sl Sci. China. Math.}, to appear.
\bibitem{CZ2} H. Cho and K. Zhu, Fock-Sobolev spaces and their Carleson measures, {\sl J. Funct. Anal.}
{\bf 263} (2012), 2483-2506.
\bibitem{C} O. Christensen, {\sl An Introduction to Frames and Riesz Basis}, Birkhauser, 2003.
\bibitem{Coburn1} L. Coburn, The Bargmann isometry and Gabor-Daubechies wavelet localization operators,
in {\sl Systems, Approximation, Singular Integral Operators, and Related Topics}, (Bordeaux 2000), 169-178;
{\sl Oper. Theory Adv. Appl.} {\bf 129}, Birkhauser, Basel, 2001.
\bibitem{Coburn2} L. Coburn, Berezin-Toeplitz quantization, in {\sl Algebraic Methods in Operator Theory},
101-108, Birkhauser, Boston, 1994.
\bibitem{CIL} L. Coburn, J. Isralowitz, and B. Li, Toeplitz operators with BMO symbols on the Segal-Bargmann
space, {\sl Trans. Amer. Math. Soc.} {\bf 363} (2011), 3015-3030.
\bibitem{DG} I. Daubechies and A. Grossmann, Frames in the Bargmann space of entire functions,
{\sl Comm. Pure. Appl. Math.} {\bf 41} (1988), 151-164.
\bibitem{F1} G. Folland, {\sl Fourier Analysis and Its Applications}, Brooks/Cole Publishing Company, 1992.
\bibitem{F2} G. Folland, {\sl Harmonic Analysis in Phase Space}, {\sl Ann. Math. Studies} {\bf 122}, Princeton
University Press, 1989.
\bibitem{G1} K. Gr\"ochenig, {\sl Foundations of Time-Frequency Analysis}, Birkh\"auser, Boston, 2001.
\bibitem{GW} K. Gr\"ochenig and D. Walnut, A Riesz basis for the Bargmann-Fock space related to sampling
and interpolation, {\sl Ark. Math.} {\bf 30} (1992), 283-295.
\bibitem{Guil} V. Guillemin, Toeplitz operators in n-dimensions, {\sl Integr. Equat. Oper. Theory} {\bf 7} (1984),
145-205.
\bibitem{H} C. Heil, History and evolution of the density theorem for Gabor frames, {\sl J. Fourier Anal. Appl.}
{\bf 13} (2007), 113-166.
\bibitem{JPR} S. Janson, J. Peetre, and R. Rochberg, Hankel forms and the Fock space,
{\sl Revista Mat. Ibero-Amer.} {\bf 3} (1987), 58-80.
\bibitem{Landau} H. Landau, Necessary density conditions for sampling and interpolation of certain entire functions,
{\sl Acta Math.} {\bf 117} (1967), 37-52.
\bibitem{Lyu} Y. Lyubarskii, Frames in the Bargmann space of entire functions, {\sl Adv. Soviet Math.} {\bf 429}
(1992), 107-113.
\bibitem{LS1} Y. Lyubarskii and K. Seip, Sampling and interpolation of entire functions and exponential systems
in convex domains, {\sl Ark. Mat.} {\bf 32} (1994), 157-194.
\bibitem{LS2} Y. Lyubarskii and K. Seip, Complete interpolating sequences for Paley-Wiener spaces and
Munckenhoupt's $A_p$ condition, {\sl Rev. Mat. Iberoamer.} {\bf 13} (1997), 361-376.
\bibitem{OS} J. Ortega-Cerd\'a and K. Seip, Fourier frames, {\sl Ann. Math.} {\bf 155} (2002), 789-806.
\bibitem{P} A. Perelomov, {\sl Generalized Coherent States and Their Applications}, Springer, Berlin, 1986.
\bibitem{Seip1} K. Seip, Reproducing formulas and double orthogonality in Bargmann and Bergman spaces,
{\sl SIAM J. Math. Anal.} {\bf 22} (1991), 856-876.
\bibitem{Seip2} K. Seip, Density theorems for sampling and interpolation in the Bargmann-Fock space I,
{\sl J. Reine Angew. Math.} {\bf 429} (1992), 91-106.
\bibitem{SW} K. Seip and R. Wallst\'en, Density theorems for sampling and interpolation in the Bargmann-Fock
space II, {\sl J. Reine Angew. Math.} {\bf 429} (1992), 107-113.
\bibitem{Z1} K. Zhu, {\sl Operator Theory in Function Spaces}, American Mathematical Society, 2007.
\bibitem{Z2} K. Zhu, {\sl Analysis on Fock Spaces}, Springer, New York, 2012.
\bibitem{Z3} K. Zhu, Singular integral operators on the Fock space, {\sl Integr. Equat. Oper. Theory}, 
{\bf 81} (2015), 451-454.


\end{thebibliography}
\end{document}